\documentclass[10pt]{amsart}
\usepackage{amssymb,amstext,amsmath,amscd,amsthm,amsfonts}
\usepackage{enumitem,graphicx,latexsym,stmaryrd,multicol,geometry,mathrsfs, mathtools}
\usepackage[normalem]{ulem}
\usepackage[usenames]{color}
\usepackage[all]{xy}
\newtheoremstyle{mystyle}
    {}
    {}
    {\itshape}
    {}
    {\bfseries}
    {.}
    { }
    {\thmname{#1}\thmnumber{ #2}\thmnote{ #3}}
\theoremstyle{mystyle}
\newtheorem{thm}{Theorem}[section]
\newtheorem{lem}[thm]{Lemma}
\newtheorem{prop}[thm]{Proposition}
\newtheorem{cor}[thm]{Corollary}
\theoremstyle{definition}
\newtheorem{dfn}[thm]{Definition}
\newtheorem{ques}[thm]{Question}

\newtheorem{rem}[thm]{Remark}
\newtheorem{conv}[thm]{Convention}

\newtheorem{ex}[thm]{Example}
\newtheorem{nota}[thm]{Notation}
\theoremstyle{remark}

\newtheorem*{claim*}{Claim}

\numberwithin{equation}{thm}
\def\A{\mathscr{A}}

\def\add{\operatorname{add}}

\def\C{\mathscr{C}}

\def\cm{\operatorname{CM}}
\def\cx{\operatorname{cx}}

\def\depth{\operatorname{depth}}
\def\dim{\operatorname{dim}}
\def\E{\mathscr{E}}

\def\Ext{\operatorname{Ext}}

\def\fpd{\operatorname{fpd}}
\def\ge{\geqslant}

\def\Hom{\operatorname{Hom}}

\def\ker{\operatorname{Ker}}
\def\le{\leqslant}

\def\m{\mathfrak{m}}

\def\Mod{\operatorname{Mod}}
\def\mod{\operatorname{mod}}

\def\pd{\operatorname{pd}}
\def\proj{\operatorname{proj}}

\def\radius{\operatorname{radius}}
\def\rap{\operatorname{rap}}
\def\res{\operatorname{res}}
\def\resp{\textit{resp}}

\def\size{\operatorname{size}}

\def\syz{\Omega}

\def\tor{\operatorname{Tor}}

\def\X{\mathscr{X}}
\def\Y{\mathscr{Y}}
\def\Z{\mathscr{Z}}
\def\ZZ{\mathbb{Z}}
\begin{document}
\allowdisplaybreaks
\title[Contravariantly infinite resolving subcategories]{Contravariantly infinite resolving subcategories}
\author{Gen Tanigawa}
\address{Graduate School of Mathematics, Nagoya University, Furocho, Chikusaku, Nagoya 464-8602, Japan}
\email{tanigawa.gen.y4@s.mail.nagoya-u.ac.jp}
\begin{abstract}
Let $R$ be a commutative Noetherian ring. Denote by $\mod R$ the category of finitely generated $R$-modules. In this paper, a contravariantly infinite subcategory of $\mod R$ is defined as a full subcategory $\X$ of $\mod R$ such that no module outside $\X$ admits a right $\X$-approximation. This paper provides several criteria for contravariant infiniteness in the case where $R$ is a local complete intersection.
\end{abstract}
\maketitle
\section{Introduction}

The notion of contravariantly finite subcategories was introduced systematically by Auslander and Smal\o  \ \cite{AS80} for artin algebras to study subcategories admitting almost split sequences. Auslander and Reiten \cite{AR91} applied it to tilting theory over an artin algebra, and this application remains an active topic of research, see \cite{AIR14,I07,IY08}. For commutative rings, the notion of contravariantly finite subcategories was initially discussed in Cohen--Macaulay approximation theory due to Auslander and Buchweitz \cite{ABu89}. Let $R$ be a Cohen--Macaulay local ring with a canonical module. Denote by $\mod R$ the category of finitely generated $R$-modules, and by $\cm R$ the full subcategory of $\mod R$ consisting of maximal Cohen--Macaulay $R$-modules. Auslander and Buchweitz proved that $\cm R$ is contravariantly finite. The notion of resolving subcategories was introduced by Auslander and Bridger \cite{ABr69} to show that the totally reflexive modules form a resolving subcategory. Takahashi \cite{T11} proved that each complete Gorenstein local ring $R$ possesses only three contravariantly finite resolving subcategories: $\mod R$, $\cm R$, and $\add R$, the full subcategory of free $R$-modules.
The purpose of this paper is to introduce and study the notion of the opposite notion of contravariantly finite subcategories, which we call contravariantly infinite subcategories. Namely, we say that a full subcategory $\X$ of an additive category $\C$ is {\em contravariantly infinite} if no object of $\C$ outside $\X$ admits a right $\X$-approximation. This notion was implicitly studied by Takahashi \cite{T21}. He proved that if $R$ is an AB ring in the sense of Huneke and Jorgensen \cite{HJ03} and $\X$ is a resolving subcategory of $\mod R$ containing a module which is not maximal Cohen--Macaulay, then $\X$ is contravariantly infinite.

Dao and Takahashi \cite{DT14} introduced the notion of a radius for full subcategories of $\mod R$ and showed that over complete local complete intersection rings with perfect coefficient field, a full subcategory $\X$ of $\mod R$ has finite radius if and only if any module in $\X$ is maximal Cohen--Macaulay. Dey, Lank and Takahashi \cite{DLT25} proved that the condition that $R$ is complete with perfect coefficient field can be replaced with the condition that $R$ is quasi-excellent. Combining these results, we shall prove the following theorem, which is the main result of this paper.

\begin{thm}\label{main}
Let $R$ be a henselian local complete intersection ring with positive dimension. Let $\X$ be a resolving subcategory of $\mod{R}$. Then the following are equivalent.
\begin{enumerate}[label=\textnormal{(\arabic*)}]
\item
$\X$ is contravariantly infinite.
\item
$\X$ contains a module of finite and positive projective dimension.
\item
$\X$ contains a module which is not maximal Cohen--Macaulay.
\end{enumerate}
\noindent
If in addition $R$ is quasi-excellent, then the above three conditions are equivalent to the following condition as well.
\begin{enumerate}[resume, label=\textnormal{(\arabic*)}]
\item
$\X$ has infinite radius.
\end{enumerate}
\end{thm}

This theorem is proved in Section 3, followed by giving basic notions and their fundamental properties in Section 2. Now it is natural to ask the following two questions. 

\begin{enumerate}
\item[(i)]
Can one remove from Theorem \ref{main} the assumption that $R$ has positive dimension?
\item[(ii)]
Can one extend Theorem \ref{main} to the Gorenstein case?
\end{enumerate}

As to (i), we get a complete answer, that is, the assumption of positive dimension is indispensable; see Remark \ref{rem} and Example \ref{ex}. As for (ii), we deeply discuss it in Section 4, introducing the subcategory of $\mod R$ consisting of modules $M$ such that the functor $\Hom_R(-,M)|_\X$ is finitely presented in the functor category of $\X$. We do not get a complete answer to (ii) itself, but instead obtain another result on complete intersections. 

\begin{thm}
Let $R$ be a henselian local complete intersection. Let $\X$ be a resolving subcategory of $\mod R$ all of whose objects are maximal Cohen--Macaulay. Let $M$ be a finitely generated $R$-module. If the functor $\Hom_R(-,M)|_\X$ is finitely generated in the functor category of $\X$, then it is finitely presented. 
\end{thm}

\section{Basic definitions and properties}

This section is devoted to stating several notations and properties for later use. We begin with the convention.

\begin{conv}
Throughout this paper, all rings are commutative Noetherian with identity, all modules are finitely generated, and all subcategories are strictly full. For an $R$-module $M$, we denote by $\pd M$ the projective dimension of $M$. We let

\begin{center}
\begin{tabular}{@{}r l@{}}
$(R,\m,k)$ & be a (commutative Noetherian) local ring, \\
$\mod{R}$  & the category of (finitely generated) $R$-modules, \\
$\cm R$   & the (full) subcategory of $\mod{R}$ consisting of maximal Cohen--Macaulay $R$-modules, and \\
$\fpd{R}$ & the (full) subcategory of $\mod{R}$ consisting of $R$-modules of finite projective dimension.
\end{tabular}
\end{center}
For simplicity, subscripts are omitted when no confusion arises.
\end{conv} 

We recall several precise definitions about closedness of subcategories. 

\begin{dfn}
Let $\A$ be an additive category and $\X$ a subcategory of $\A$.
\begin{enumerate}[label=\textrm{(\arabic*)}]
\item
We say that $\X$ is {\em closed under finite direct sums} if for a finite number of objects $X_1,\dots,X_n \in \X$, the direct sum $X_1 \oplus \cdots \oplus X_n$ belongs to $\X$.
\item
We say that $\X$ is {\em closed under direct summands} provided that if $A$ and $B$ are objects in $\A$ with $A\oplus B\in\X$, then both $A$ and $B$ belong to $\X$.
\end{enumerate}
\end{dfn}

We denote by $\add_{\A}{\X}$ the additive closure of $\X$, that is, the smallest subcategory of $\A$ that contains $\X$ and is closed under finite direct sums and direct summands. When $\X$ consists of a single module $X$, we simply write $\add_{\A}X$.

\begin{dfn}
Let $\A$ be an abelian category and $\X$ a subcategory of $\A$.
\begin{enumerate}[label=\textrm{(\arabic*)}]
\item
We say that $\X$ is {\em closed under subobjects} ({\em resp. closed under quotient objects}) if $A \in \X$ for all exact sequences $0 \to A\to B$ (\resp. $0 \gets A\gets B$) in $\A$ with $B \in \X$.
\item
We say that $\X$ is {\em closed under kernels} ({\em resp. closed under cokernels}) if $A \in \X$ for all exact sequences $0 \to A\to B\to C$ (\resp. $0 \gets A\gets B\gets C$) in $\A$ with $B,C \in \X$.
\item
We say that $\X$ is {\em closed under kernels of epimorphisms} ({\em resp. closed under cokernels of monomorphisms}) if $A \in \X$ for all exact sequences $0 \to A\to B\to C\to 0$ (\resp. $0 \gets A\gets B\gets C\gets 0$) in $\A$ with $B,C \in \X$.
\item
We say that $\X$ is {\em closed under extensions} if $B \in \X$ for all exact sequences $0 \to A\to B\to C\to 0$ in $\A$ with $A,C \in \X$.
\end{enumerate}
\end{dfn}

\begin{dfn}
Let $\A$ be an abelian category with enough projective objects.
\begin{enumerate}[label=\textrm{(\arabic*)}]
\item We denote by $\proj{\A}$ the subcategory of $\A$ consisting of projective objects.
\item Let $M$ be an object of $\A$. For each positive integer $n$, we denote by $\syz^n_{\A}{M}$ the $n$-th {\em syzygy} of $M$, that is, the object $N$ of $\A$ that appears in an exact sequence $0 \to N\to P_{n-1} \to \cdots \to P_0\to M\to 0$ in $\A$ with $P_i \in \proj{\A}$ for all $0 \leq i \leq n-1$. The object $\syz^n_{\A}{M}$ is uniquely determined up to projective summands. We set $\syz^0_{\A}M=M$ and $\syz_{\A}M = \syz^1_{\A}M$. For each $n\ge0$ we denote by $\syz^n \A$ the subcategory of $\A$ consisting of $n$-th syzygies.
\item We say that a subcategory $\X$ of $\A$ is {\em closed under syzygies} if $\X$ contains $\proj\A$ and $\syz_{\A}{X} \in \X$ for all $X \in \X$. 
\end{enumerate}
\end{dfn}

\begin{rem}
Let $\A$ be an abelian category and $\X$ a subcategory of $\A$. Consider the conditions that $\X$ is closed under
\begin{enumerate}[label=\textrm{(\arabic*)}]
\item subobjects,
\item kernels,
\item kernels of epimorphisms,
\item quotient objects,
\item cokernels,
\item cokernels of monomorphisms,
\item direct summands and
\item extensions.
\end{enumerate}
\noindent
Then the implications $(1) \Rightarrow (2) \Rightarrow (3)$, $(4) \Rightarrow (5) \Rightarrow (6)$ and  $(2) \Rightarrow (7) \Leftarrow (5)$ hold. In addition, suppose that $\A$ has enough projective objects. Consider the condition that $\X$ is closed under
\begin{enumerate}[resume]
\item syzygies.
\end{enumerate}
\noindent
If $\X$ contains $\proj\A$ and satisfies (7) and (8), then (3) and (9) are equivalent.
\end{rem}

Now, we are ready to recall the definition of a resolving subcategory, which was introduced by Auslander and Bridger \cite{ABr69}. In most of the situations in this paper, we deal with resolving subcategories of $\mod R$.
\begin{dfn}
Let $\A$ be an abelian category with enough projective objects. A subcategory $\X$ of $\A$ is called {\em resolving} if $\X$ contains $\proj{\A}$ and is closed under direct summands, extensions and kernels of epimorphisms. We denote by $\res_{\A}{\X}$ the resolving closure of $\X$, that is, the smallest resolving subcategory of $\A$ that contains $\X$.
\end{dfn}

We recall the definitions of a right approximation and a contravariantly finite subcategory, which were introduced by Auslander and Smal\o\ \cite{AS80}. Takahashi \cite{T11} proved that all contravariantly finite resolving subcategories of $\mod R$ over a henselian Gorenstein local ring $R$ are $\add R$, $\cm R$ and $\mod R$.

\begin{dfn}
Let $\A$ be an additive category and $\X$ a subcategory of $\A$.
\begin{enumerate}[label=\textrm{(\arabic*)}]
\item
We say that a morphism $f:X \to M\ (\resp.\ f:X \leftarrow M)$ in $\A$ with $X \in \X$ is a {\em right} (\resp. {\em left}) {\em $\X$-approximation} (of $\A$) if for every morphism $f':X' \to M\ (\resp.\ f':X' \leftarrow M)$ with $X' \in \X$ there exists a morphism $g:X' \to X\ (\resp.\ g:X' \leftarrow X)$ such that $f' = fg\ (\resp.\ f' = gf)$.
\item
We denote by $\rap_{\A}{\X}$ the subcategory of $\A$ consisting of objects that admit right $\X$-approximations. For each object $X \in \X$ the identity morphism of $X$ is a right $\X$-approximation, so one has
$$\X \subseteq \rap{\X} \subseteq \A.$$
\item
We say that $\X$ is {\em contravariantly finite} if every object in $\A$ admits a right $\X$-approximation, that is, if the equality $\rap{\X} = \A$ holds.
\end{enumerate}
\end{dfn}

For a resolving subcategory $\X$ of $\mod R$, the subcategory $\rap\X$ of $\mod R$ is closed under direct summands and extensions, see \cite[Proposition 23(1)]{T21}. We freely use this fact.

For a subcategory $\X$ of $\mod{R}$, we denote by $\X^{\bot}$ (\resp. $^{\bot}\X$) the subcategory of $\mod{R}$ consisting of $R$-modules $M$ such that $\Ext^{>0}_R{(X,M)}=0$ (\resp. $\Ext^{>0}_R{(M,X)}=0$) for all $X \in \X$. When $\X$ consists of a single module $X$, we simply write $X^{\bot}$ (\resp. $^{\bot}X$). Let $\Y$ be a subcategory of $\mod R$ such that $\Y\subseteq\X$. Then it is straightforward that $\X^{\bot}$ (\resp. $^{\bot}\X$) is contained in $\Y^{\bot}$ (\resp. $^{\bot}\Y$). The following lemma is used to ensure that there exists some nonzero $R$-module in $\X^{\bot}$ if $\X \neq \rap{\X}$.

\begin{lem}[{\cite[Lemma 14]{T21}}]\label{1}
Let $R$ be a henselian local ring and $\X$ a resolving subcategory of $\mod{R}$. Then an $R$-module $M$ belongs to $\rap \X$ if and only if there exists an exact sequence $0 \to Y\to X\to M\to 0$ of $R$-modules with $X \in \X$ and $Y \in \X^\bot$. 
\end{lem}

Now we recall the definition of the radius of a subcategory of $\mod R$, which was introduced by Dao and Takahashi \cite{DT14}. The radius indicates the least number of extensions necessary to build all the $R$-modules in $\X$ out of a single $R$-module.

\begin{dfn}
Let $\X$ and $\Y$ be subcategories of $\mod R$.
\begin{enumerate}
\item We denote by [$\X$] the additive closure of the subcategory of $\mod R$ consisting of $R$ and $\syz^i X$ for all $i \in \ZZ_{\ge 0}$ and $X \in \X$.
\item We denote by $\X \circ \Y$ the subcategory of $\mod R$ consisting of $R$-modules $M$ which appears in an exact sequence $0 \to X\to M\to Y\to 0$ with $X \in \X$ and $Y \in \Y$. We set $\X \bullet \Y = [[\X]\circ[\Y]]$.
\item We define the {\em ball of radius r centered at} $\X$  as
$$
[\X]_r =
\begin{cases}
[\X] & (r=1),\\
[\X]_{r-1} \bullet \X & (r>1).\\
\end{cases}
$$
When $\X$ consists of a single module $X$, we simply write $[X]_r$ instead of $[\X]_r$, and call it the {\em ball of radius r centered at X}.
\item We define the {\em radius} of $\X$, denoted by $\radius \X$, as the infimum of the integers $n \ge 0$ such that $\X$ is contained in $[C]_{n+1}$ for some $R$-module $C$.
\end{enumerate}
\end{dfn}

Finally, we present several important results which play a key role in the proof of the main theorem. The notion of AB rings was introduced by Huneke and Jorgensen \cite{HJ03}. Let $R$ be a Gorenstein ring. If there exists a nonnegative integer $n$ such that $\Ext^{>n}_R(M,N)=0$ for all $M,N \in \mod R$ satisfying $\Ext^{\gg 0}_R(M,N)=0$, we say that $R$ is an {\em AB} ring. A typical example of an AB ring is a complete intersection; see \cite[Corollary 3.4]{HJ03}. The following theorem gives a sufficient condition for a resolving subcategory to be contravariantly infinite.

\begin{thm}[{\cite[Proposition 29(2)]{T21}}]\label{2}
Let $R$ be a henselian AB ring and $\X$ a resolving subcategory of $\mod{R}$. If $\X$ is not contravariantly infinite, then $\X$ is contained in $\cm R$.
\end{thm}

The following theorem is powerful to classify the resolving subcategories of $\mod R$ and is one of the reasons why the main theorem of this paper demands the assumption that $R$ is a complete intersection ring.

\begin{thm}[{\cite[Theorem 7.4]{DT15}}]\label{3}
Let $R$ be a complete intersection ring. There exists a 1-1 correspondence
$$
\begin{Bmatrix}
\text{Resolving}\\
\text{subcategories}\\
\text{of }\mod R
\end{Bmatrix}
\begin{matrix}
 & \overset{\Phi}{\longrightarrow} & \\
 & \underset{\Psi}{\longleftarrow} & 
\end{matrix}
\begin{Bmatrix}
\text{Resolving}\\
\text{subcategories}\\
\text{of }\fpd R
\end{Bmatrix}
\times
\begin{Bmatrix}
\text{Resolving}\\
\text{subcategories}\\
\text{of }\cm R
\end{Bmatrix}
.$$
\noindent
Here, the mutually inverse bijections $\Phi,\Psi$ are defined by $\Phi(\X) = (\X \cap \fpd R,\X \cap \cm R)$ and $\Psi(\Y,\Z) = \res (\Y \cup \Z)$.
\end{thm}

Thanks to the following two theorems, we observe that the resolving subcategory $\X$ of $\mod R$ has finite radius if and only if $\X$ is contained in $\cm R$ over a complete, complete intersection ring $R$ with perfect coefficient field.

\begin{thm}[{\cite[Theorem I]{DT14}}]\label{4}
Let $R$ be a complete intersection ring and $\X$ a resolving subcategory of $\mod R$. If $\X$ has finite radius, then $\X$ is contained in $\cm R$.
\end{thm}

\begin{thm}[{\cite[Theorem I\!I]{DT14}}]\label{5}
Let $R$ be a Cohen--Macaulay complete ring with perfect coefficient field. Then $\cm R$ has finite radius.
\end{thm}

\begin{rem}\label{6}
In Theorem \ref{5}, the condition that $R$ is complete with perfect coefficient field can be replaced with the condition that $R$ is quasi-excellent with a canonical module $\omega$. Let us show this in the following. For a subcategory $\X$ of $\mod R$, we put $|\X|=\add\X$. The {\em size} of a subcategory is defined by replacing ``[-]" with ``$|$-$|$" in the definition of the radius. Since $|\X|_n$ is contained in $[\X]_n$ for all $n>0$, we have $\radius\X \le \size\X$. Let $R$ be a quasi-excellent Cohen--Macaulay ring with a canonical module $\omega$. According to \cite[Corollary 3.12]{DLT25}, there exists an $R$-module $G$ such that $\syz^s(\mod R)$ is contained in $|G|_{n+1}$ for some $s,n\ge 0$. In the proof of \cite[Corollary 5.9]{DT14}, replacing ``$d$'' with ``$s$'', we obtain $\size\cm R<\infty$. Hence $\cm R$ has finite radius. 
\end{rem}

\section{Proof of the main theorem}

In this section, we prove the main theorem step by step. Throughout this section, we let $\X$ be a resolving subcategory of $\mod R$. We show three lemmas.

\begin{lem}\label{3.2}
Let $R$ be a henselian ring. Suppose that $\X$ is not contravariantly infinite. Then the equality $\X \cap \fpd R= \add R$ holds.
\end{lem}

\begin{proof}
By assumption, there exists an $R$-module $M$ that belongs to $\rap{\X}$ but not to $\X$. By Lemma \ref{1}, we obtain a short exact sequence $0 \to Y\to X\to M\to 0$ with $X\in \X$, $Y\in \X^{\bot}$ and $Y\neq 0$. Suppose that $\X \cap \fpd{R}$ strictly contains $\add{R}$. Then there exists an $R$-module $Z \in \X$ of finite positive projective dimension, say $n$. Hence we obtain a short exact sequence $0 \to R^{\oplus{p}} \stackrel{A}{\to} R^{\oplus{q}} \to \Omega^{n-1}Z \to 0$ such that $p,q \in \ZZ_{>0}$ and any entry of the matrix $A$ is in $\m$. Applying $\Hom_R{(-,Y)}$ to it, we get an exact sequence $Y^{\oplus{q}} \xrightarrow{{}^\mathrm{t}A} Y^{\oplus{p}} \to \Ext^1_R{(\Omega^{n-1}Z,Y)}$, where ${}^\mathrm{t}A$ denotes the transpose of the matrix $A$. The Ext module vanishes because $\Omega^{n-1}Z$ is in $\X$ and $Y$ is in $\X^{\bot}$. So we have $Y^{\oplus{p}} = \m Y^{\oplus{q}}$. Nakayama's lemma implies $Y=0$, which is a contradiction.
\end{proof}

\begin{lem}\label{3.3}
Suppose that $R$ is Gorenstein with $d=\dim R>0$. If all the modules in $\X$ are maximal Cohen--Macaulay, then $\X$ is not contravariantly infinite.
\end{lem}

\begin{proof}
It suffices to show that $\rap\X$ contains an $R$-module which does not belong to $\X$. Since $R$ is Cohen--Macaulay, we have $\depth{R} = d > 0$. Hence there exists an $R$-regular element $x$ of $\m$. There is a short exact sequence
\begin{equation}\label{3.3.1}
0 \to R\stackrel{x}{\to} R\to R/(x) \to 0.
\end{equation}
\noindent
As $\depth{R/(x)} = d-1$, the $R$-module $R/(x)$ is not maximal Cohen--Macaulay. By assumption, $R/(x)$ does not belong to $\X$. The module $R$ belongs to both $\X^{\bot}$ and $\X$ since $\X$ is contained in $\cm R$ and resolving. Lemma \ref{1} implies that $R/(x)$ admits a right $\X$-approximation, that is, $R/(x)$ belongs to $\rap{\X}$.
\end{proof}

\begin{lem}\label{3.4}
Let $R$ be a henselian AB ring with $\dim R>0$. Then $\X$ is contained in $\cm{R}$ if and only if $\X$ is not contravariantly infinite.
\end{lem}

\begin{proof}
The ``only if'' part follows from Lemma \ref{3.3}, while the ``if'' part is shown by Theorem \ref{2}.
\end{proof}

Now we are ready to prove our main result.

\begin{thm}\label{3.5}
Let $R$ be a henselian complete intersection ring with $\dim R>0$. Then the following conditions are equivalent.
\begin{enumerate}[label=\textnormal{(\arabic*)}]
\item
The subcategory $\X$ is not contravariantly infinite.
\item
The subcategory $\X$ contains no module of finite positive projective dimension.
\item
All modules in $\X$ are maximal Cohen--Macaulay.
\end{enumerate}
\noindent
If in addition $R$ is quasi-excellent, then the above three conditions are equivalent to the following condition as well.
\begin{enumerate}[resume,label=\textnormal{(\arabic*)}]
\item
The subcategory $\X$ has finite radius.
\end{enumerate}
\end{thm}

\begin{proof}
The implication (1) $\Rightarrow$ (2) follows from Lemma \ref{3.2}. If $\X \cap \fpd{R} = \add{R}$ holds, then we obtain the following equalities by Theorem \ref{3}.
\begin{align*}
    \X &= \Psi \Phi(\X) = \Psi(\X \cap \fpd R,\X \cap \cm R) = \Psi(\add R,\X \cap \cm R)\\
    &= \res(\add R\cup (\X \cap \cm R)) = \res(\X\cap\cm R)\subseteq\res\cm R=\cm R.
\end{align*}
Therefore the implication (2) $\Rightarrow$ (3) holds. The equivalence (1) $\Leftrightarrow$ (3) (\resp. (3) $\Leftrightarrow$ (4)) is a consequence of Lemma \ref{3.4} (\resp. Theorem \ref{4} and Remark \ref{6}).
\end{proof}

\begin{rem}\label{rem}
It is shown by Auslander, Ding and Solberg \cite{ADS} that complete intersection rings satisfy the following condition, which is called the {\em Auslander--Reiten condition}.
\vspace{5pt}
\begin{center}
\begin{tabular}{@{}l l@{}}
(\textbf{ARC}) & For every $R$-module $M$, if $\Ext_R^{>0}(M,M \oplus R) =0$, then $M$ is free.
\end{tabular}
\end{center}
\vspace{5pt}
\noindent
We say that an $R$-module $M$ has {\em complexity c} if $c$ is the least nonnegative integer $n$ such that there exists a real number $\alpha$ such that the $i$-th Betti number of $M$ is bounded above by $\alpha i^{n-1}$ for all $i\gg 0$. We denote by $\cx_R M$ the complexity of $M$. Let $R$ be an Artinian complete intersection ring. Then the statements (2) and (3) in Theorem \ref{3.5} hold obviously, while (4) follows from the proof of \cite[Theorem 5.7]{DT14}. However the example below indicates the necessity of the assumption in Theorem \ref{3.5} that $\dim R$ is positive.
\end{rem}

\begin{ex}\label{ex}
Let $R$ be an Artinian complete intersection, $c$ the codimension of $R$, and $\E$ the subcategory of $\mod R$ consisting of $R$-modules whose complexity are less than $c$. Then $\E$ is resolving, see \cite[Proposition 4.2.4]{Av98}. Now we consider the case where $c \ge 2$. Assume that $\E$ is not contravariantly infinite. Then there exists an exact sequence
$$
0 \to M \to E \to L \to 0
$$
of $R$-modules with $L\notin\E$, $E\in\E$ and $M \in \E^{\bot}$. Suppose that $\cx_R M$ is less than $c$. Then $M$ is in $\E$. Since $R$ satisfies (\textbf{ARC}), $M$ is free. As $R$ is Artinian Gorenstein, the above exact sequence splits, and we get an isomorphism $E\cong M\oplus L$. This implies that $L$ belongs to $\E$, which is a contradiction. So we have $\cx_R M =c$. Then, according to \cite[Proposition 2.7]{CDT14}, $M$ is a {\em test module}, that is, all $R$-modules $N$ with $\tor_{\gg 0}^R(M,N) =0$ (this is equivalent to $\Ext_R^{\gg 0}(N,M)=0$, see \cite[Theorem I\!I\!I]{AvBu00}) have finite projective dimension. Thus $\E = \add R$, so that $c=1$. This contradicts the fact that $c\ge 2$. As a result, $\E$ is contravariantly infinite.
\end{ex}

\section{A question related to the main theorem}

In this section, we discuss the following question that is motivated by Theorem \ref{main}.

\begin{ques}
Does Theorem \ref{main} hold for all Gorenstein rings?
\end{ques}

Throughout this section, we fix the following notation.

\begin{nota}
Let $(R,\m,k)$ be a henselian Gorenstein local ring and $\X$ a resolving subcategory of $\mod R$.
\end{nota}

It is well known that $(\cm R,\fpd R)$ is a {\em cotorsion pair}, that is, $(\cm R)^{\bot} = \fpd R$ and $\cm R=\hspace{1pt} ^{\bot}(\fpd R)$, due to \cite{Isc69}. This fact is used in the proof of Proposition \ref{x}.

First, we recall the definitions of the functor category of an additive category, finitely generated objects and finitely presented objects, which were introduced by Auslander \cite{A66}. 

\begin{dfn}
Let $\A$ be an additive category.
\begin{enumerate}
\item We denote by $\Mod \A$ the {\em functor category} of $\A$. The objects of $\Mod \A$ are the additive contravariant functors from $\A$ to the category of abelian groups. The morphisms of $\Mod \A$ are the natural transformations.
\item We say that an object $F \in \Mod \A$ is {\em finitely generated} if there exists an exact sequence $\Hom_{\A}(-,A_0)\to F\to 0$ in $\Mod \A$ with $A_0 \in \A$, and that $F$ is {\em finitely presented} if there exists an exact sequence $\Hom_{\A}(-,A_1) \to \Hom_{\A}(-,A_0)\to F\to 0$ in $\Mod\A$ with $A_0,A_1 \in \A$. 
\item We denote by $\mod \A$ the subcategory of $\Mod \A$ consisting of the finitely presented objects of $\Mod \A$.
\end{enumerate}
\end{dfn}

Applying the lemma below to $\mod R$, we have that for an additive subcategory $\X$ of $\mod R$, an $R$-module $M$ belongs to $\rap \X$ if and only if $\Hom_R(-,M)|_{\X}$ belongs to $\mod \X$.

\begin{lem}[{\cite[Lemma 16]{T21}}]\label{c}
Let $\A$ be an additive category and $\X$ an additive subcategory of $\A$. An object $A \in \A$ admits a right $\X$-approximation if and only if the functor $\Hom_{\A}(-,A)|_{\X}$ is finitely generated.
\end{lem}

Let $\C$ be a subcategory of $\mod R$ consisting of $R$-modules $C$ such that $\Hom_R(-,C)|_{\X}$ is in $\mod \X$, that is, $\Hom_R(-,C)|_{\X}$ is finitely presented. Obviously, the inclusions 
$$\X \subseteq \C \subseteq \rap \X$$
hold. From now on, we shall consider what properties the subcategory $\C$ satisfies and when these two inclusions are equalities. We begin with recalling the following lemma.

\begin{lem}[{\cite[Lemma 24]{T21}}]\label{a}
Let $0\to L\to M\to N\to 0$ be an exact sequence of $R$-modules. If $L \in \X^{\bot} \cap \C$ and $M \in \C$, then $N \in \C$.
\end{lem}

The exact sequence (\ref{3.3.1}) implies that $R/(x) \in \C$. Hence Lemma \ref{3.3} is refined as the following form.

\begin{lem}\label{h}
Let $R$ be a Gorenstein ring with $\dim R>0$. If $\X$ is contained in $\cm R$, then $\C$ contains an $R$-module outside $\X$.
\end{lem}

Now we have the following observation. 

\begin{prop}\label{?}
Let $C$ be an $R$-module in $\C$. Then there exists an exact sequence $0\to K\to X\to C\to0$ of $R$-modules with $X\in\X$ and $K\in\rap\X$.
\end{prop}

\begin{proof}
There exists an exact sequence
$$\Hom_R(-,X_1)|_{\X}\to\Hom_R(-,X_0)|_{\X} \xrightarrow{\phi}\Hom_R(-,C)|_{\X}\to 0$$
in $\mod \X$ with $X_0,X_1 \in \X$. By Yoneda's lemma, there exists an $R$-homomorphism $f:X_0\to C$ such that $\phi=\Hom_R(-,f)|_{\X}$. Decomposing this sequence, we get exact sequences 
\begin{align*}
    &\Hom_R(-,X_1)|_{\X}\to \Hom_R(-,K)|_{\X}\to0,\\
    &0\to\Hom_R(-,K)|_{\X}\to\Hom_R(-,X_0)|_{\X}\to\Hom_R(-,C)|_{\X}\to0
\end{align*}
in $\Mod \X$, where $K=\ker f$. Lemma \ref{c} implies $K\in\rap\X$. There exists an induced exact sequence $0\to K\to X_0\xrightarrow{f} C\to 0$ of $R$-modules.
\end{proof}

We will use the following proposition to obtain several properties of $\C$.

\begin{prop}\label{e}
If there exists an exact sequence $0\to N\to M\to C\to 0$ of $R$-modules with $M \in \rap \X$ and $C \in \C$, then $N \in \rap \X$.
\end{prop}

\begin{proof}
We have an exact sequence $0\to K\to X\to C\to 0$ of $R$-modules with $X\in \X$ and $K\in\rap\X$ by Proposition \ref{?}. There is a pullback diagram.
$$
\begin{CD}
 @. @. 0 @. 0 @.\\
@. @. @AAA @AAA @.\\
0 @>>> N @>>> M @>>> C @>>> 0\\
@. @| @AAA @AAA @.\\
0 @>>> N @>>> U @>>> X @>>> 0\\
@. @. @AAA @AAA @.\\
 @. @. K @= K @.\\
@. @. @AAA @AAA @.\\
 @. @. 0 @. 0 @.
\end{CD}
$$
\noindent
As the $R$-modules $M$ and $K$ belong to $\rap\X$, so does $U$ by the middle column in the above diagram. There is an exact sequence $0 \to \syz X\to F\to X\to 0$ of $R$-modules with $F \in \add R$. We get the following pullback diagram.
$$
\begin{CD}
 @. @. 0 @. 0 @.\\
@. @. @AAA @AAA @.\\
0 @>>> N @>>> U @>>> X @>>> 0\\
@. @| @AAA @AAA @.\\
0 @>>> N @>>> U' @>>> F @>>> 0\\
@. @. @AAA @AAA @.\\
 @. @. \syz{X} @= \syz{X} @.\\
@. @. @AAA @AAA @.\\
 @. @. 0 @. 0 @.
\end{CD}
$$
Since $F$ is free, the middle row in the second diagram splits. Therefore $N$ is a direct summand of $U'$. As the $R$-modules $U$ and $\syz X$ belong to $\rap \X$, so does $U'$ by the middle column in the second diagram. Hence $N$ belongs to $\rap \X$.
\end{proof}

The following corollary is similar to Lemma \ref{1}, which will be used to ensure that there exists some nonzero $R$-module that belongs to both $\X$ and $\X^{\bot}$ if there exists an $R$-module in $\C$ outside $\X$.

\begin{cor}\label{g}
An $R$-module $C$ belongs to $\C$ if and only if there exists an exact sequence $0\to Y\to X\to C\to 0$ of $R$-modules such that $X$ belongs to $\X$ and $Y$ belongs to both $\rap\X$ and $\X^{\bot}$. When these equivalent conditions hold, there exists an exact sequence $X_1\to X_0\to C\to0$ of $R$-modules with $X_0\in\X$ and $X_1\in\X\cap\X^\perp$ which induces an exact sequence
$$
\Hom_R(-,X_1)|_\X\to\Hom_R(-,X_0)|_\X\to\Hom_R(-,C)|_\X\to0
$$
in $\mod\X$.
\end{cor}

\begin{proof}
The ``only if'' part follows from Lemma \ref{1} and Proposition \ref{e}. It suffices to show the ``if'' part. Since  $Y$ belongs to $\rap\X$, there exists an exact sequence $0\to Y'\to X'\to Y\to0$ of $R$-modules such that $X'\in\X$ and $Y\in\X^\perp$. Then $X'$ belongs to $\X\cap\X^\perp$ and the induced sequence $\Hom_R(-,X')|_{\X}\to \Hom_R(-,Y)|_{\X}\to 0$ in $\Mod\X$ is exact. On the other hand, we have an induced exact sequence $0\to \Hom_R(-,Y)|_{\X}\to \Hom_R(-,X)|_{\X}\to \Hom_R(-,C)|_{\X}\to 0$ in $\Mod\X$. Combining these two sequences, we obtain an exact sequence $\Hom_R(-,X')|_{\X}\to \Hom_R(-,X)|_{\X}\to \Hom_R(-,C)|_{\X}\to 0$ in $\mod\X$. Hence $C$ belongs to $\C$. 
\end{proof}

\begin{rem}
If there exists an exact sequence $\cdots\to \Hom_R(-,Y_2)|_\X\to \Hom_R(-,Y_1)|_\X\to \Hom_R(-,Y_0)|_\X\to \Hom_R(-,C)|_\X\to 0$ in $\mod\X$ with $Y_i\in\X$ for all $i$, then we similarly obtain an exact sequence
$$\cdots \to \Hom_R(-,X_2)|_{\X} \to \Hom_R(-,X_1)|_{\X} \to \Hom_R(-,X_0)|_{\X} \to \Hom_R(-,C)|_{\X} \to 0$$
in $\mod\X$ with $X_0 \in \X$ and $X_i \in \X \cap \X^{\bot}$  for all $i>0$.
\end{rem}

The following corollary indicates how close $\C$ and $\rap\X$ are.

\begin{cor}\label{NCM}
    Let $0\to N\to C\to M\to 0$ be an exact sequence of $R$-modules. If $C\in \C$ and $M,N\in \rap\X$, then $M\in\C$.
\end{cor}

\begin{proof}
We have an exact sequence $0\to Y\to X\to M\to 0$ of $R$-modules with $X\in\X$ and $Y\in\X^\bot$ by Lemma \ref{1}. There is a pullback diagram.
$$
\begin{CD}
 @. @. 0 @. 0 @.\\
@. @. @VVV @VVV @.\\
 @. @. N @= N @.\\
@. @. @VVV @VVV @.\\
0 @>>> Y @>>> V @>>> C @>>> 0\\
@. @| @VVV @VVV @.\\
0 @>>> Y @>>> X @>>> M @>>> 0\\
@. @. @VVV @VVV @.\\
 @. @. 0 @. 0 @.
\end{CD}
$$
As the $R$-modules $N$ and $X$ belong to $\rap\X$, so does $V$ by the middle column in the diagram. Since the $R$-module $C$ is in $\C$, Proposition \ref{e} implies that $Y$ belongs to $\rap\X$ by the middle row in the diagram. Hence $M$ belongs to $\C$ by Corollary \ref{g}.
\end{proof}

It turns out that $\C$ has the following closedness properties, as does $\rap\X$.

\begin{cor}
    The subcategory $\C$ of $\mod R$ is closed under direct summands and extensions.
\end{cor}

\begin{proof}
If $M\oplus N$ belongs to $\C$, then $M$ and $N$ are in $\rap\X$ since $\C$ is contained in $\rap\X$. Corollary \ref{NCM} implies that $M$ and $N$ belong to $\C$. Hence $\C$ is closed under direct summands. We observe from Corollary \ref{g} and the same argument as in the proof of \cite[Proposition 23(1)]{T21} that $\C$ is closed under extensions.
\end{proof}

Now we consider the {\em generalized Auslander--Reiten condition}.
\vspace{5pt}
\begin{center}
\begin{tabular}{@{}l l@{}}
(\textbf{GARC}) & If for each $R$-module $M$ there exists a nonnegative integer $n$ such that\\
& $\Ext_R^{>n}(M,M \oplus R) =0$, then the projective dimension of $M$ is $n$ or less.
\end{tabular}
\end{center}
\vspace{5pt}
AB rings satisfy (\textbf{GARC}), see \cite[Corollary 2.8]{Ar19}. The following proposition may be viewed as a generalization of Theorem \ref{2}.

\begin{prop}\label{x}
Let $R$ be a Gorenstein ring with $\dim R>0$. Assume that $R$ satisfies \textnormal{(\textbf{GARC})}. Then $\X$ is strictly contained in $\C$ if and only if $\X$ is contained in $\cm R$.
\end{prop}

\begin{proof}
The ``if'' part is a consequence of Lemma \ref{h}. In what follows, we show the ``only if'' part. By Corollary \ref{g}, there exists a nonzero $R$-module $X$ that belongs to both $\X$ and $\X^{\bot}$. The condition (\textbf{GARC}) and Lemma \ref{3.2} imply that $X$ is in  $\add R$. Since $R$ is local and $\X^{\bot}$ is closed under direct summands, $R$ is a direct summand of $X$ and belongs to $\X^{\bot}$. Furthermore, $\X^{\bot}$ is closed under extensions and cokernels of monomorphisms, so that $\fpd R$ is contained in $\X^{\bot}$. We observe that ${}^{\bot}(\fpd R)$ contains ${}^{\bot}(\X^{\bot})$. It is clear that $\X$ is contained in $^{\bot}(\X^{\bot})$. Since $\cm R$ coincides with ${}^\bot (\fpd R)$, $\X$ is contained in $\cm R$.
\end{proof}

Finally, we consider the condition that $\C = \rap \X$. A resolving subcategory $\X$ satisfies this condition if and only if for each $R$-module $M$ the functor $\Hom_R(-,M)|_\X$ is finitely presented whenever it is finitely generated. Thus we name such a resolving subcategory as follows.

\begin{dfn}
We say that $\X$ is {\em coherent} if $\C$ coincides with $\rap \X$.
\end{dfn}

We do not know any example of a non-coherent resolving subcategory. We can relate the coherence of a resolving subcategory with several other conditions as follows.

\begin{prop}\label{aaa}
Consider the following three conditions.
\begin{enumerate}[label=\textnormal{(\arabic*)}]
\item
The resolving subcategory $\X$ is coherent.
\item
Every module in $\X^\perp$ admits a right $\X$-approximation.
\item
The subcategory $\X^\perp$ is closed under syzygies.
\item
The subcategory $\rap\X$ is closed under kernels of epimorphisms.
\end{enumerate}
Then the implications \textnormal{(4)} $\Leftrightarrow$ \textnormal{(1)} $\Leftarrow$ \textnormal{(2)} $\Leftarrow$ \textnormal{(3)} hold.
\end{prop}

\begin{proof}
(1) $\Leftrightarrow$ (4): The ``only if'' part follows from \cite[Proposition 23(5)]{T21}, while the ``if'' part is shown by Proposition \ref{e}.

(2) $\Rightarrow$ (1): Let $M$ be an $R$-module in $\rap \X$. By Lemma \ref{1}, we obtain an exact sequence $0 \to Y\to X\to M\to 0$ with $X \in \X$ and $Y \in \X^{\bot}$. By assumption, the $R$-module $Y$ is in $\rap \X$. Hence $M$ belongs to $\C$ thanks to Corollary \ref{g}.

(3) $\Rightarrow$ (2): Let $Y$ be an $R$-module in $\X^{\bot}$. Consider the exact sequence $0 \to \Omega Y\to F\to Y\to 0$ with $F\in\add R$. By assumption, the $R$-module $\Omega Y$ belongs to $\X^{\bot}$. Hence $Y$ belongs to $\rap \X$ by Lemma \ref{1}.
\end{proof}

The implication (4) $\Rightarrow$ (1) in Proposition \ref{aaa} immediately gives rise to the following result. 

\begin{cor}
Let $\X$ be a resolving subcategory of $\mod R$. Suppose that $\X$ is either contravariantly finite or contravariantly infinite. Then $\X$ is coherent. 
\end{cor}

We recall the condition (\textbf{A}) in the sense of Mifune \cite[Definition 3.8]{M25}. Let $\A$ be an abelian category and $\E$ a subcategory of $\A$.
\vspace{5pt}
\begin{center}
\begin{tabular}{@{}l l@{}}
(\textbf{A}) & For every object of $X\in\E$ there exists an exact sequence $0\to X\to Y\to X'\to 0$\\
  & in $\A$ such that $X'$ belongs to $\E$ and $Y$ belongs to both $\E$ and $\E^{\bot}$.
\end{tabular}
\end{center}
\vspace{5pt}

\begin{thm}\label{bbb}
Assume that $R$ is Gorenstein and satisfies \textnormal{(\textbf{ARC})}. Suppose that $\X$ is contained in $\cm R$ and satisfies \textnormal{(\textbf{A})}. Then $\X$ is coherent. 
\end{thm}

\begin{proof}
In view of Proposition \ref{aaa}, it suffices to show that $\X^\perp$ is closed under syzygies. Let $M$ be an $R$-module in $\X^\perp$, and take an exact sequence $0\to\syz M\to F\to M\to0$ of $R$-modules with $F$ free. Pick any $R$-module $X$ in $\X$. As $\Ext_R^{>0}(X,R)=0$, we see that $\Ext_R^i(X,\syz M)\cong\Ext_R^{i-1}(X,M)=0$ for all $i>1$. Since $\X$ satisfies (\textbf{A}), there is an exact sequence $0\to X\to Y\to X'\to0$ of $R$-modules with $X'\in\X$ and $Y\in\X\cap\X^\perp$. Then $\Ext_R^{>0}(Y,Y\oplus R)=0$. As $R$ satisfies (\textbf{ARC}), the $R$-module $Y$ is free. It is observed that $\Ext_R^1(X,\syz M)\cong\Ext_R^2(X',M)=0$. Thus $\syz M$ belongs to $\X^\perp$, as desired. 
\end{proof}

Applying the above theorem, we obtain the following sufficient condition for $\X$ to be coherent. 

\begin{cor}
Assume that $R$ is a complete intersection and $\X$ is contained in $\cm R$. Then $\X$ is coherent. 
\end{cor}

\begin{proof}
If $R$ is a complete intersection, then $R$ is Gorenstein and satisfies (\textbf{ARC}). Since $\X$ is contained in $\cm R$, it follows from \cite[Corollary 4.16]{DT14} that every $R$-module $X\in\X$ admits an exact sequence $0\to X\to F\to X'\to0$ of $R$-modules such that $F$ is free and $X'$ is in $\X$. Since any free $R$-module belongs to $\X\cap\X^\perp$, the resolving subcategory $\X$ satisfies the condition (\textbf{A}). Now Theorem \ref{bbb} implies that $\X$ is coherent.
\end{proof}

\section*{Acknowledgments}
The author would like to thank his supervisor Ryo Takahashi for his careful reading of this manuscript and many thoughtful suggestions. He also thanks Kaito Kimura, Shinnosuke Kosaka, Yuki Mifune, and Yuya Otake for their valuable comments.

\bibliographystyle{abbrv}
\bibliography{ForCIRS}
\end{document}